\documentclass[10pt,reqno]{amsart}
\usepackage{a4wide}
\usepackage{fancyhdr}
\usepackage{amsfonts,amssymb,amsmath,amsthm,latexsym,indentfirst}
\usepackage{epsfig}
\usepackage{enumerate}
\usepackage{latexsym}
\usepackage[pagewise]{lineno}
\usepackage{mathtools}
\usepackage{amstext}
\mathtoolsset{showonlyrefs}
\usepackage{graphicx}
\usepackage{hyphenat}

\usepackage{textcomp}
\usepackage{graphicx}
\usepackage{color}

\numberwithin{equation}{section}

\newtheorem{theorem}{Theorem}[section]
\newtheorem{definition}[theorem]{Definition}
\newtheorem{lemma}[theorem]{Lemma}
\newtheorem{proposition}[theorem]{Proposition}
\newtheorem{remark}[theorem]{Remark}

\newcommand{\C}[1][]{\ensuremath{{\mathbb{C}^{#1}} }}
\newcommand{\R}[1][]{\ensuremath{{\mathbb{R}^{#1}} }}

\newcommand{\dis}{\displaystyle}

\newcommand{\re}{\mathbb{R}}

\begin{document}

\title[Global solutions to INLS]{On global well-posedness, scattering and other properties for infinity energy solutions to inhomogeneous NLS Equation}

\author[ Mykael Cardoso, Roger P. de Moura and Gleison N. Santos]{Mykael Cardoso, Roger P. de Moura and Gleison N. Santos}

\address{Universidade Federal do Piau\'i, Campus Universit\'ario Ministro Petrônio Portella, Ininga, 64049-550, Teresina, Piau\'i
, Brazil}
\email{mourapr@ufpi.edu.br}

\address{Universidade Federal do Piau\'i, Campus Universit\'ario Ministro Petrônio Portella, Ininga, 64049-550, Teresina, Piau\'i
	, Brazil}
\email{mykael@ufpi.edu.br}

\address{Universidade Federal do Piau\'i, Campus Universit\'ario Ministro Petrônio Portella, Ininga, 64049-550, Teresina, Piau\'i
, Brazil}
\email{gleison@ufpi.edu.br}





%






\maketitle
\begin{center}
   \begin{minipage}[t]{0.4\linewidth}
        \centering
     {\small\textit{  In celebration of 50 years of the Mathematics Department at UFPI.}}
    \end{minipage}
    \end{center}
\begin{abstract}
In this work, we consider the inhomogeneous nonlinear Schrödinger (INLS) equation in $\mathbb{R}^n$
\begin{align}
i\partial_t u + \Delta u + \gamma |x|^{-b}|u|^{\alpha} u = 0,
\end{align}
where $\gamma=\pm 1$, and $\alpha$ and $b$ are positive numbers. Our main focus is to estabilish the global well-posedness of the INLS equation in Lorentz spaces for $0<b<2$ and $\alpha<\frac{4-2b}{N-2}$. To achieve this, we use Strichartz estimates in Lorentz spaces $L^{r,q}(\R^n)$ combined with a fixed point argument. Working on Lorentz space setting instead the classical $L^p$ is motivated by the fact that the potential $|x|^{-b}$ does not belong the usual $L^p$-space. As a consequence of the ideas developed here on the global solution study we obtain some other properties for INLS, such as, existence of self-similar solutions, scattering, wave operators and assymptotic stability.  
\end{abstract}

\section{Introduction}

We consider the initial value problem (IVP) to the inhomogeneous nonlinear Schr\"odinger (INLS) equation
\begin{equation} \label{pvi1}
\left\{ \begin{array}{lc}
i\partial_t u + \Delta u + \gamma |x|^{-b}|u|^{\alpha} u = 0, \ (x, t) \in \mathbb{R}^n \times \mathbb{R}, \\
u(x, 0)=\varphi(x), 
\end{array}
\right.
\end{equation}
where $u=u(x, t)$ is a complex-valued function and $\varphi$ is a tempered distribution, $n\geq 1$ and $0<b<2$. The integral equation corresponding to \eqref{pvi1} is 
\begin{equation} \label{EqInt1}
u(t)=e^{it\Delta}\varphi+i\gamma\displaystyle\int_0^te^{i(t-\tau)\Delta}|x|^{-b}(|u|^\alpha u)(\tau)d\tau,
\end{equation}
where $e^{it\Delta}$ is the linear Schr\"odinger group. 

The inhomogeneous nonlinear Schr\"odinger equation is a mathematical model that has received significant attention in recent years due to its relevance in various physical settings and its interesting borderline case from a mathematical point of view. In recent years, there has been a growing interest in studying the well-posedness and scattering theories for the INLS equation. One of the central questions in this field is to establish the global well-posedness of the IVP for the INLS equation. A number of authors have made significant contributions in this direction (see e.g. \cite{C21}, \cite{CF21}, \cite{Farah}, \cite{CARLOS}, \cite{M21}). Note that when $b=0$ in the INLS model, we have the classical nonlinear Schrödinger equation, which has been extensively studied over the last three decades (see Bourgain \cite{Bourgain}, Cazenave \cite{CAZENAVEBOOK}, Linares-Ponce \cite{FELGUS}, Tao \cite{Tao}, and references therein).

Due to Hamiltonian structure of the problem, the solutions for the INLS equation enjoys the followings laws of conservation called mass and energy, respectively
\begin{align}
    M[u(t)]=\int_{\re^n}|u(x,t)|^2\,dx=M[u_0]
\end{align}
and
\begin{align}
    E[u(t)]=\frac{1}{2}\int_{\re^n}|\nabla u(x,t)|^2\,dx-\frac{\gamma}{\alpha+2}\int_{\re^n}|x|^{-b}|u(x,t)|^{\alpha+2}\,dx=E[u_0],
\end{align}
provided $u$ has sufficient regularity.

The INLS equation is invariant under the scaling $u_\mu(t,x)=\mu^{p}u(\mu x, \mu^2 t),  \mu >0,$ where $Re\,p=\frac{2-b}{\alpha}$. This means that if $u$  is a solution of \eqref{pvi1}, with initial data $\varphi$, so is $u_\mu$ with initial data $$u_{\mu,0}=\mu^{p}\varphi(\mu x).$$  A straightforward computation yields to
$$
\|u_{0,\mu}\|_{\dot{H}^s}=\mu^{s-\frac{n}{2}+\frac{2-b}{\alpha}}\|\varphi\|_{\dot{H}^s},
$$
implying that the scale-invariant Sobolev space is $\dot{H}^{s_b}(\mathbb{R}^n)$, with $s_b=\frac{n}{2}-\frac{2-b}{\alpha}$, the so called \textit{critical Sobolev index}. For $s_b = 0$ (or $\alpha = \frac{4-2b}{n}$) the IVP \eqref{pvi1} is known as mass-critical or $L^2$-critical; if $s_b<0$ (or $0<\alpha<\frac{4-2b}{n}$) it is called mass-subcritical or $L^2$-subcritical; if $0<s_b<1$, \eqref{pvi1} is known as mass-supercritical and energy-subcritical (or intercritical).

We briefly review the literature about the well-posedness of  \eqref{pvi1}.  It was first studied by Genoud-Stuart \cite{GENSTU}, where they used the energy method to show the local well-posedness in $H^1(\mathbb{R}^n)$ for the $H^1$-subcritical case, $n\geq 1$ and $0<b<\min\{n,2\}$. They also established global well-posedness in the mass-subcritical case. In the mass-critical case, Genoud in \cite{GENOUD} showed global well-posedness in $H^1(\mathbb{R}^n)$, provided that the mass of the initial data is below that of the associated ground state. This result was extended in the case $\frac{4-2b}{n}<\alpha<\frac{4-2b}{n-2}$ by Farah in \cite{Farah}.

On the other hand, Guzmán in \cite{CARLOS}, using the contraction mapping principle, obtained local well-posedness in $H^1(\mathbb{R}^n)$ for the energy subcritical case in dimensions $n\geq 4$ and $0<b<2$. Cho-Lee \cite{Cho-Lee} treated the case $n=3$ for $0<b<\frac{3}{2}$ and Dinh \cite{Dinh2017} the case $n=2$ for $0<b<1$. Furthermore, in the intercritical case, Guzmán in \cite{CARLOS} also established a small data global theory in $H^1(\mathbb{R}^n)$ for $n\geq 4$, Campos in \cite{C21} treated the case $n=3$, and Cardoso, Farah and Guzmán \cite{CFG20} the case $n=2$. In all these works, the range of $b$ is the same where local well-posedness was obtained. Recentely, Campos, Correira and Farah \cite{CCF22}  established the local well-posedness for intercritical INLS \eqref{pvi1} in homogeneous Sobolev spaces $\dot H^s$ for $n \geq 1$, $0 \leq s \leq 1$ such that $s < n/2$, and $0 < b <\min\{n/2 + 1 - s, n - s, 2\}.$

The problem of the existence of infinity energy global solutions for the nonlinear Schrödinger equation was first studied by Cazenave and Weissler in \cite{CW}, where they investigate the existence of global solutions for the NLS equation (IVP \eqref{pvi1} with $b=0$) for small initial data $\varphi$ with respect to the norm $\|\varphi\|_{L^{\alpha+2}}=\sup_{t>0}t^{\beta_0}\|e^{it\Delta}\varphi\|_{L^{\alpha+2}}$ where $\beta_0=\frac{4-\alpha(n-2)}{2\alpha(\alpha+2)}$. The difficulty in that case lies in determining which functions have this finite norm. They then consider solutions whose initial data are homogeneous to obtain self-similar solutions. In that article, it was explicitly calculate  the evolution of the flow of the linear Schrödinger equation for functions of the form $\varphi(x)=|x|^{-p}$. Other authors have also studied the same problem in the context of the NLS equation (see, for example, \cite{CW00},  \cite{Zhang03},  \cite{P00cauchy}, \cite{ribaud1998regular}). In \cite{BFV09} Braz e Silva, Ferreira and Villamizar-Roa extend the results of Cazenave and Weissler to include global existence and uniqueness, as well as global self-similar solutions, in the context of Lorentz spaces.

Here we intend to extend these results to the INLS setting. Our results are obtained in the range
\begin{align}\label{alpha00}
\alpha_0<\alpha<2^*_b,
\end{align}
where 
\begin{align}\label{alpha0}
2^*_b=\left\{\begin{array}{cc}
    \frac{4-2b}{n-2}, &  if \ \ n\geq 3\\
     \infty, & if \ \ n\leq 2
\end{array}
\right.
\end{align}
and  $\alpha_0$ is the positive root of the equation
$$n\alpha^2+(n-2+2b)\alpha-4+2b=0.$$
For $\alpha$ satisfying the condition \eqref{alpha00}, we define the parameter $\beta$ by
\begin{align}\label{beta}
    \beta=\frac{4-2b-\alpha(n-2)}{2\alpha(\alpha+2)}.
\end{align}
Note that 
\begin{align}\label{nalpha}
\beta(\alpha+1)<1,\,\,\,\frac{n\alpha}{2(\alpha+2-b)}<1
\end{align}
and
\begin{align}
    \beta+1-\frac{n\alpha+2b}{2(\alpha+2)}-\beta(\alpha+1)=0.
\end{align}
Our first result is about the global well-posedness for \eqref{pvi1}. In particular, here we extend the global well-posedness results of Cazenave and Weissler \cite{CW} and Braz e Silva, Ferreira and Villamizar-Roa \cite{BFV09} to the inhomogeneous case $0<b<2.$

\begin{theorem}\label{globalWP} Consider $\alpha$ satisfying \eqref{alpha00}, $r=\frac{n(\alpha+2)}{n-b}$, $q>r$ and let $\beta$ be given by \eqref{beta}. Suppose further that $\rho>0$ and $M>0$ satisfy the inequality
$$\rho +KM^{\alpha+1}\leq M,$$
with $K=K(\alpha, n,\gamma)$ given by
\begin{align}\label{condK}
    K=2|\gamma| (\alpha+1)(4\pi)^{-\frac{n\alpha}{2(\alpha+2)}}B\left(1-\frac{n\alpha+2b}{2(\alpha+2)}, 1-\beta(\alpha+1)\right),
\end{align}
where $B(\cdot,\cdot)$ is the beta function. Let $\varphi$ be a tempered distribution such that 
\begin{align}\label{cond1}
    \sup_{t>0}|t|^\beta\|e^{it\Delta}\varphi\|_{L^{r,q}}\leq \rho.
\end{align}
Then there exists a unique positively global solution $u$ to \eqref{EqInt1} such that 
\begin{align}\label{cond2}
    \sup_{t>0}|t|^{\beta} \|u(t)\|_{L^{r,q}}\leq M.
\end{align}
\end{theorem}

The term $|x|^{-b}$ presents a challenge, as it does not belong to usual $L^p$ spaces. Therefore, it is needed to use an appropriate space where we can obtain integrability of this potential. Our study on the existence of solutions for the Schrödinger equation relies on the analysis of the Schrödinger linear group $e^{it\Delta}$ within the framework of Lorentz spaces. 

A natural question about the assumptions of the previous theorem is whether there exist initial data that satisfy such assumptions. In fact, there are, and those are homogeneous initial data as already considered by Cazenave and Weissler \cite{CW00} in the settings of the NLS equation.

Considering homogeneous initial data is a natural choice. To see this we consider the dilation operator $D_\lambda=D_{\lambda,p}$ given by 
\begin{align}
D_{\lambda,p} \varphi(x)= \lambda^{p}\varphi(\lambda x),
\end{align}
where $\lambda > 0$ and $p$ is a fixed complex power such that $0 < Re(p) < n$. From direct computations, we can check that (see Lemma \ref{Dil})
\begin{align}
\|D_{\lambda}\varphi\|_{L^{r,q}} = \lambda^{Re(p)-\frac{n}{r}}\|\varphi\|_{L^{r,q}}
\end{align}
and
\begin{align}
e^{it\Delta}=D_\lambda e^{i\lambda^2t\Delta}D_{\lambda^{-1}}.
\end{align}
Thus, fixed $t>0$, $\lambda=t^{-\frac{1}{2}}$ and $Re(p)=\frac{2-b}{\alpha}$ we have
\begin{align}\label{evohom}
    t^{\beta}\|e^{it\Delta}\varphi\|_{L^{\frac{n(\alpha+2)}{n-b},q}}&=t^{\frac{4-2b-\alpha(n-2)}{2\alpha(\alpha+2)}}\|D_{t^{-\frac{1}{2}}}e^{i\Delta}D_{t^{\frac{1}{2}}}\varphi\|_{L^{\frac{n(\alpha+2)}{n-b},q}}\\
    &=t^{\frac{4-2b-\alpha(n-2)}{2\alpha(\alpha+2)}}t^{-\frac{2-b}{2\alpha}+\frac{n-b}{2(\alpha+2)}}\|e^{i\Delta}D_{t^{\frac12}}\varphi\|_{L^{\frac{n(\alpha+2)}{n-b},q}}\\&=\|e^{i\Delta}D_{t^{\frac12}}\varphi\|_{L^{\frac{n(\alpha+2)}{n-b},q}}.
\end{align}
With this we have an important class of solutions to be considered which are invariant by dilatation $D_{\lambda}$ for all $\lambda>0$ called $p$-homogeneous functions. This motivates the following definition.

\begin{definition}
A solution $u(x, t)$ of equation \eqref{EqInt1} is considered self-similar if for $p$ such that $Re\,p = \frac{2 - b}{\alpha}$ and for all $\lambda > 0$, $$u(x,t)=\lambda^{p}u(\lambda x,\lambda^2t).$$
\end{definition}

Thus, as a consequence of Theorem \ref{globalWP}, we prove that if the initial data $\varphi$ is homogeneous of degree $\frac{2-b}{\alpha}$ and if $\|e^{i\Delta}\varphi\|_{L^{r,q}}<\varepsilon$ for $q>r>0$ and $\varepsilon>0$ sufficiently small, then the corresponding solution $u$ satisfies 
\begin{align}
    u(x,t)=\lambda^{p}u(\lambda x, \lambda^2t),
\end{align}
where $Re(p) = \frac{2-b}{\alpha}$ . More precisely, we show the following result.

\begin{theorem}[Self-similar solutions]\label{selfsimilar}Assume \eqref{alpha0} , $r=\frac{n(\alpha+2)}{n-b}$, $q>r$ and suppose $Re\,p=\frac{2-b}{\alpha}$. If $\varphi $ is a finite linear combination of functions of the form $P_k|x|^{-p-k}$, where $P_k$ is a homogeneous harmonic polynomial of degree $k$ (including $k=0$), then $\|e^{i\Delta}\varphi \|_{L^{r}}$ is finite and  
 \begin{align}
|t|^\beta\|e^{it\Delta}\varphi\|_{L^{r}}=\|e^{i\Delta}\varphi\|_{L^{r}},\,\,\,\forall t>0,
 \end{align}
 where $\beta $ is given by \eqref{beta}. In addition, if $\|e^{i\Delta}\varphi\|_{L^{r}}$ is sufficiently small, there exists a self-similar solution $u$ to \eqref{EqInt1} with initial data $\varphi$, having all the proprieties described in Theorem \ref{globalWP}.
\end{theorem}

To see the novelty of this work, it is important to note that all known results concerning the global well-posedness of solutions to this model in the intercritical case ($\frac{4-2b}{n}<\alpha<2^*_b$) have been established in the non-homogeneous Sobolev spaces $H^s$ or in the homogeneous spaces $\dot H^s$ with $0\leq s\leq 1$. However, there exist solutions, such as self-similar solutions, that do not belong to these spaces. In fact, if we assume that $u$ is a self-similar solution to equation \eqref{pvi1} with initial data $u_0$, we can show that $\|u(t)\|_{L^2}$ and $E[u(t)]$ are not conserved quantities, which contradicts the conservation laws of mass and energy for the equation in \eqref{pvi1} if we assume $\alpha_0<\alpha<2_b^*$ and  that $u(t) \in L^2$ or $u(t) \in \dot H^1 \cap L^{\alpha+2}(|x|^{-b})$. To be more specific, consider the following calculation
\begin{align*}
\|u_0\|_{L^2}& = \|\lambda^{p}u(\lambda x, \lambda^2t)\|_{L^2}= \lambda^{\frac{2-b}{\alpha}}\|u(\lambda x, \lambda^2t)\|_{L^2}=\lambda^{-\frac{n}{2}+\frac{2-b}{\alpha}}\|u( \lambda^2t)\|_{L^2}=\lambda^{-s_b}\|u_0\|_{L^2}
\end{align*}
and
\begin{align}
E[u_0]=E[\lambda^{p}u(\lambda x, \lambda^2 t)]=\lambda^{1-s_b}E[u_0],
\end{align}
for all $\lambda>0$, where $s_b=\frac{n}{2}-\frac{2-b}{\alpha}$. In particular, this implies that if $\alpha_0<\alpha<2^*_b$, then the corresponding self-similar solution with initial data $u_0$ admits infinite energy and is not in the $H^s$ space. In the case of homogeneous spaces $\dot H^s$, global well-posedness is allowed only if $s=s_b$ . This can be seen by a similar calculation
\begin{align*}
\|u(t)\|_{\dot H^s}& = \|t^{-\frac{p}{2}}u(t^{-\frac12} x, 1)\|_{\dot H^s}\\&= t^{-\frac{2-b}{2\alpha}}\|u(t^{-\frac12} x, 1)\|_{\dot H^s}\\&=t^{-\frac{s}{2}+\frac{n}{4}-\frac{2-b}{2\alpha}}\|u(1)\|_{\dot H^s}\\&=t^{\frac{s_b-s}{2}}\|u(1)\|_{\dot H^s}.
\end{align*}

\begin{remark}
   \end{remark}
    \begin{enumerate}
    {\it
        \item[$(1)$] One way to study self-similar solutions to nonlinear evolution equations is by analyzing the corresponding nonlinear elliptic equations. In particular, for a self-similar solution $u$ to equation \eqref{pvi1}, we can express it as $u(x, t)=t^{-\frac{p}{2}}f\left(t^{-\frac12}x\right)$, where $f=u(\cdot,1)$. To analyze $f$, we obtained the corresponding elliptic equation given by
        \begin{equation}\label{EqEl}
        \Delta v - \frac{i}{\alpha} v - \frac{i}{2} x \cdot \nabla v = \gamma |x|^{-b} |v|^{\alpha} v.
        \end{equation}
where $v=f(\cdot)$. This approach has been used in previous studies, such as \cite{kw1}, \cite{kw2}, in the case of the classical NLS. However, solving the elliptic equation \eqref{EqEl} is generally a challenging task.

        \item[$(2)$] Planchon \cite{P00cauchy} established the global well-posedness for the nonlinear Schrödinger equation (case $b=0$) in Besov spaces $\dot B^{s_0,\infty}$ for $\alpha>4/n$ and $\alpha\in 2\mathbb N\backslash {0}$. In particular, he obtained the existence of self-similar solutions and, due to the continuous embedding of $\dot H^{s_0}$ into $\dot B^{s_0,\infty}$, the well-posedness in the critical Sobolev space $\dot H^{s_0}$. This result was later improved by Miao-Zhang-Zhang \cite{Zhang03} for the case where $\alpha>4/n$ and $\alpha\notin 2\mathbb N$. We believe that similar results can be obtained for the INLS configuration and we will address this in a future work.}
       
    \end{enumerate}
    
Once global results are proved, the natural route is to study the asymptotic behavior of such global solutions as $t \to \infty$. We prove that our solutions scatter to a solution of the linear problem in $L^{r,q}$. Moreover, we construct the wave operator associated with equation \eqref{pvi1}. This corresponds to the reciprocal problem in scattering theory, which involves constructing a solution with a prescribed scattering state. Inspired in Cazenave and Weissler \cite{CW00} we show the following results.
\begin{theorem}[Scattering]\label{teo scattering}
    Suppose $\alpha_0<\alpha<2^*_b$. Let $\mathcal X$ and $\mathcal W$ be the spaces defined in \eqref{class init} and \eqref{W}, respectively. If $u\in \mathcal{X}$ is a solution of \eqref{pvi1}, then there exists $\psi_+ \in\mathcal{W}$ such that 
    \begin{align}
        \|e^{-it\Delta}u(t)-\psi_+\|_{L^{r,q}}\leq Ct^{-\beta}|\!|\!| u |\!|\!|^{\alpha+1},
    \end{align}
    for all $t>0$, where $| \! | \! | u | \! | \! | : = \displaystyle\sup_{t > 0} t^{\beta}\| u(t) \|_{L^{r, q}}.$
\end{theorem}
\begin{theorem}[Wave operator]\label{teowo} Assume $\rho, M>0$ satisfy the inequality
\begin{align}
    \rho+KM^{\alpha+1}\leq M,
\end{align}
    where $K=K(\alpha,\gamma)$ is given by \eqref{condK}. Let $\psi\in \mathcal{W}$ be such that $|\!|\!| e^{it \Delta}\psi|\!|\!|\leq \rho$. Then there exists a unique solution $u\in \mathcal X$ of \eqref{pvi1} such that $|\!|\!|u|\!|\!|\leq M$ and 
    \begin{align}
        \|e^{-it\Delta}u(t)-\psi\|_{L^{r,q}}\to 0
    \end{align}
    as $t\to \infty.$
\end{theorem}

Furthermore, we examine the asymptotic stability of the global solutions and prove that regular perturbations of the linear Schrödinger equation have negligible effects over long periods. 

\begin{theorem}\label{asymptsta} Let $0 \leq h < 1 - \beta (\alpha + 1)$, $r=\frac{n(\alpha+2)}{n-b}$, $q>r$ and $\phi, \varphi \in L^{r, q}(\mathbb{R}^n)$ satisfying
\begin{equation}\label{Lasym}
\lim_{t \rightarrow + \infty} t^{\beta + h}\|e^{it\Delta} (\phi - \varphi)\|_{L^{r,q}} = 0.
\end{equation}
Let $u, v$ be the global solutions provided by Theorem \ref{globalWP} corresponding to the initial data $\phi$ and $\varphi$ respectively such that
\begin{equation}\label{condasymp}
(2M)^{\alpha}B\left(1 -\frac{n}{2}\left( \frac{1}{r'} - \frac{1}{r}\right), 1-\beta(\alpha + 1) - h\right)
\end{equation}
is sufficiently small. Then
\begin{equation}\label{NLasym}
\lim_{t \rightarrow + \infty} t^{\beta + h}\| u(t) - v(t)\|_{L^{r,q}} = 0.
\end{equation}
\end{theorem}

The paper is organized as follows. In Section \ref{sec3}, we introduce some notations and give a review on Lorentz spaces and its properties. The Section \ref{sec4} is dedicated to proof the global well-posedness for equation \eqref{pvi1} applying the contraction mapping principle. In Section \ref{sec5} we  discuss the existence of self-similar solutions for equation \ref{pvi1}. The Section \ref{sec7} is devoted to show the scattering and wave operator results. We finish proving the asymptotic stability solutions for \eqref{pvi1}.

%
%



\section{Notation and preliminary estimates}\label{sec3} 

Given any positives constants $C,$ $D,$ by
$C\lesssim D$ we mean that there exists a constant $c>0$ such that
$C\leq cD$; and, by $C\sim D$ we mean $C\lesssim D$ and $D\lesssim
C$. $L^{p}(\re^n)$-norms will be written as $\parallel\cdot\parallel_{L^{p}}$.

\subsection{Lorentz spaces}
For $0< p, q \leq \infty$ we define the Lorentz space $L^{p, q}(\mathbb{R}^n)$ as the space of all measurable functions $f:\mathbb{R}^n\to\mathbb{C}$ such that 
$$
\| f \|_{L^{p, q}} := \left\{\begin{array}{rl}
	\left(\dfrac{q}{p}\displaystyle \int_0^{+ \infty} \left[t^{\frac{1}{p}} f^{\ast}(t)\right]^q \dfrac{dt}{t}\right)^{\frac{1}{q}}, & \mbox{if } \ \ q < \infty\\
	& \\
	\displaystyle\sup_{t > 0} \ t^{\frac{1}{p}} f^{\ast}(t), & \mbox{if } \ \ q = \infty
	\end{array}
\right.
$$
is finite, where
$$f^{\ast}(t):=\inf\{\,\lambda>0\,;\,d_f(\lambda)\leq t\,\},$$
with $d_f(\lambda)=|\{\,x\in\re^{n}\,;\;|f(x)|>\lambda\,\}|$
being the distribution function of $f$ on $(0,\infty)$.

It is known that $L^{p,q}$ are quasinormed and quasi-Banach  spaces for all $0<p,q\leq\infty$. Nevertheless for  $1< p<\infty$ and $1\leq q\leq\infty$ they can be turned into Banach spaces, endowed with the norm (cf. Bergh $\&$ Löfström \cite{BL} and O'Neil \cite{ON})
$$
\| f \|'_{L^{p, q}} 
:= 
\left\{
\begin{array}{rl}
	\left(\dfrac{q}{p}\displaystyle \int_0^{+ \infty} \left[t^{\frac{1}{p}} f^{\ast \ast}(t)\right]^q \dfrac{dt}{t}\right)^{\frac{1}{q}}, & 1< p<\infty, \ \  1 \leq q < \infty\\
	& \\
	\displaystyle\sup_{t > 0} \ t^{\frac{1}{p}} f^{\ast \ast}(t), & 1 < p \leq \infty,\ \  q = \infty
	\end{array}
\right.
$$
where
$$f^{\ast \ast}(t) := \dfrac{1}{t}\int_0^{\infty} f^{\ast}(s) ds.$$
Since $\| f \|'_{L^{p, q}} \sim\| f \|_{L^{p, q}}$  we shall use $\| f \|_{L^{p, q}} $ to perform our estimates.

\begin{remark}
    It is worth to note that $L^{p}\hookrightarrow L^{p,q}$ if $p\leq q.$
\end{remark}

The following generalized Hölder inequality will be crucial to the success of our approach on global well-posedness.

\begin{lemma}\label{l1p} Let $1<p_1,p_2<\infty$ be such that $\frac1{p_1}+\frac1{p_2}<1$,  and consider $1\leq q_1,q_2\leq\infty$. If $f\in L^{p_1,q_1}(\R^n)$ and $g\in L^{p_2,q_2}(\R^n)$ then  $fg\in L^{p,q}(\R^n)$, where
	$$\dfrac{1}{p} = \dfrac{1}{p_1} + \dfrac{1}{p_2}, \ \ \dfrac{1}{q} \leq \dfrac{1}{q_1} + \dfrac{1}{q_2}.$$
 Moreover, there exists a constant $C=C(p,q,,p_1,q_1,p_2,q_2)$ such that
	$$\| f g\|_{L^{p, q}} \leq C\| f \|_{L^{p_1, q_1}}\|g\|_{L^{p_2, q_2}}.$$
\end{lemma}

\begin{proof}
See O'Neil \cite{ON}, Theorem 3.4.   
\end{proof}

\begin{lemma}\label{l2p}
	Let $0 < b < n$. Then $|x|^{-b} \in L^{\frac{n}{b}, \infty}(\mathbb{R}^n)$.
\end{lemma}

\begin{proof}
It follows directly from the definition. 
\end{proof}

\begin{lemma}\label{l22p} For every $a > 0$ it holds
	$$\left\| |f|^{a}\right\|_{L^{p, q}}=\| f \|_{L^{ap, aq}}^{a}.$$
	\end{lemma}
	
	\begin{proof}
In fact, it follows from the definition and the fact that $(|f|^a)^*=(f^*)^a$ (see Remark 1.4.7 in Grafakos \cite{Graf}). 
\end{proof}

\begin{lemma}\label{Dil} If $\delta_a(f)(x)=f(ax)$ denotes a dilatation by $a > 0$, then
	$$\left\|\delta_{a}(f)\right\|_{L^{p, q}}=a^{-n/p}\| f \|_{L^{p, q}}.$$
	\end{lemma}
	
	\begin{proof}
Since $(\delta_a(f))^*(t)=f^*(a^nt)$, it is enough to apply the definition of Lorentz spaces. 
\end{proof}

\begin{lemma}\label{l3p}
Let $1 \leq q \leq + \infty$ and $2 < p < \infty$. If
$$\frac{1}{p} + \frac{1}{p'} = 1$$
then there exists a constant $C > 0$ such that
$$\|e^{it\Delta} \varphi \|_{L^{p, q}(\mathbb{R}^n)} \leq C |t|^{-\frac{n}{2}\left(\frac{1}{p'} - \frac{1}{p}\right)}\| \varphi \|_{L^{p'\!, q}(\mathbb{R}^n)}$$
for all $t \neq 0$.
\end{lemma}

\begin{proof}
See Lemma 2.1 in \cite{BFV}.  
\end{proof}
The following result due to Cazenave and Weissler \cite{CW} will be crucial to prove the existence of self-similar solutions to the INLS equation.
\begin{proposition}[\cite{CW}]\label{propcaz} Let $P_k$ be a homogenous harmonic polynomial of degree $k$ and $\omega(x)=P_k(x)|x|^{-k}$. Consider $\varphi=\omega(x)|x|^{-p}$ where $0 < Re\,p < n$.
Then, $e^{it\Delta}\varphi \in L^r(\mathbb R^n)$ for all $t > 0$ and all $r$ such that
$$r> max\left\{\frac{n}{Re\,p},\frac{n}{n-Re\,p}\right\}.$$
\end{proposition}

\section{Global well-posedness theory on a Lorentz space: Proof of Theorem \ref{globalWP}}\label{sec4}
We remind  
$$r = \dfrac{n(\alpha + 2)}{ n - b} \ \ {\rm and}\ \  \beta = \dfrac{4-2b-\alpha(n-2)}{2\alpha(\alpha+2)}.$$
We consider
\begin{equation}\label{class init}
\mathcal X=\left\{ u : \mathbb{R}^n \times \mathbb{R}\to\C \ | \ \sup_{t > 0} t^{\beta}\| u(t) \|_{L^{r, q}} < \infty\right\}
\end{equation}
endowed with the norm
$$| \! | \! | u | \! | \! | : = \sup_{t > 0} t^{\beta}\| u(t) \|_{L^{r, q}},$$
and 
\begin{equation}\label{W}
\mathcal W=\left\{ \varphi : \mathbb{R}^n\to\C \ | \ e^{it\Delta} \varphi\in\mathcal X\right\}.
\end{equation}
For $M > 0$ we define
$$\mathcal X_M := \left\{u : \mathbb{R}^n \times \mathbb{R} \ \longrightarrow \ \mathbb{C}\ | \ \sup_{t > 0} t^{\beta}\| u(t)  \|_{L^{r, q}} \leq M\right\}.$$

Fix $\varphi\in\mathcal W$ and consider the integral operator associated to the problem \eqref{pvi1}
$$\Psi(u) = e^{it\Delta} \varphi +i \gamma \int_0^t e^{i(t-\tau)\Delta}(|x|^{-b}|u|^{\alpha}u)(\tau) d \tau.$$
Following the well known contraction principle argument we shall prove that for a suitable $M$ we have
\begin{equation}\label{welldef}
	\Psi (X_M) \subseteq X_M
	\end{equation}
and
\begin{equation}\label{contraction}
	|\!|\!|\Psi (u) - \Psi (v)|\!|\!| \leq c |\!|\!|u - v|\!|\!|\ \ \forall \ u, v \in X_M, \ 
\end{equation}	
for some $0 < c < 1$.

Let then $u \in X_M$. We shall estimate the norm $|\!|\!|\Psi (u)|\!|\!|$. Using Lemma \ref{l3p} we have
\begin{align}\label{eq4gwp}
	\left\| \int_0^t e^{i(t-\tau)\Delta}(|x|^{-b}|u|^{\alpha}u)(\tau) d \tau\right\|_{L^{r,q}} &\leq 	\int_0^t \left\| e^{i(t-\tau)\Delta}(|x|^{-b}|u|^{\alpha}u)(\tau) \right\|_{L^{r,q}} d \tau \nonumber\\
	&\lesssim \int_0^t (t - \tau )^{-\frac{n}{2}\left(\frac{1}{r'} - \frac{1}{r}\right)} \left\| |x|^{-b}|u|^{\alpha}u \right\|_{L^{r'\!,q}} d \tau.
\end{align}
Now, since
$$\dfrac{1}{r'} = \dfrac{\alpha + 1}{r} + \dfrac{b}{n}\ \ {\rm and}\ \  \dfrac{1}{q} \leq \dfrac{\alpha + 1}{q} + \dfrac{1}{\infty},$$
it follows from Lemma \ref{l1p} that
\begin{align}\label{eq5gwp}
	\left\| |x|^{-b}|u|^{\alpha}u \right\|_{L^{r'\!,q}} &\lesssim \left\| |u|^{\alpha}u \right\|_{L^{\frac{r}{\alpha + 1}, \frac{q}{\alpha + 1}}} \| |x|^{-b}\|_{L^{\frac{n}{b}, \infty}}.
	\end{align}
Using Lemma \ref{l2p} and Lemma \ref{l22p} in \eqref{eq5gwp} we get 
\begin{align}\label{eq6gwp}
	\left\| |x|^{-b}|u|^{\alpha}u \right\|_{L^{r', q}} &\lesssim \left\| u \right\|_{L^{r,q}}^{\alpha + 1} \| |x|^{-b}\|_{L^{\frac{n}{b}, \infty}}\nonumber\\
	&\lesssim \left\| u \right\|_{L^{r,q}}^{\alpha + 1}.
\end{align}
Replacing estimate \eqref{eq6gwp} into \eqref{eq4gwp} we find
\begin{equation}\label{eq7gwp}
	\left\| \int_0^t e^{i(t-\tau)\Delta}(|x|^{-b}|u|^{\alpha}u)(\tau) d \tau\right\|_{L^{r,q}} 
	\lesssim \int_0^t (t - \tau )^{-\frac{n}{q}\left(\frac{1}{r'} - \frac{1}{r}\right)} \left\| u (\tau )\right\|_{L^{r,q}}^{\alpha + 1} d \tau.
\end{equation}
Also, we can bound the right hand side of \eqref{eq7gwp} as
\begin{align*}\label{eq9gwp}
\int_0^t (t - \tau )^{-\frac{n}{2}\left(\frac{1}{r'} - \frac{1}{r}\right)} \left\| u (\tau )\right\|_{L^{r,q}}^{\alpha + 1} d \tau&
\leq \sup_{\tau > 0} \left\{\tau^{\beta}\left\| u (\tau )\right\|_{L^{r,q}}\right\}^{\alpha + 1} \int_0^t \tau^{-(\alpha + 1) \beta}(t - \tau )^{-\frac{n}{2}\left(\frac{1}{r'} - \frac{1}{r}\right)} d \tau\\
& \lesssim M^{\alpha + 1}\int_0^t \tau^{-(\alpha + 1) \beta}(t - \tau )^{-\frac{n}{2}\left(\frac{1}{r'} - \frac{1}{r}\right)} d \tau
\end{align*}
and so
\begin{equation}\label{eq9gwp}
	\left\| \int_0^t e^{i(t-\tau)\Delta}(|x|^{-b}|u|^{\alpha}u)(\tau) d \tau\right\|_{L^{r,q}} 
	\lesssim M^{\alpha + 1}\int_0^t \tau^{-(\alpha + 1) \beta}(t - \tau )^{-\frac{n}{2}\left(\frac{1}{r'} - \frac{1}{r}\right)} d \tau .
\end{equation}
On the other hand
\begin{align}\label{eq10gwp}
\int_0^t \tau^{-(\alpha + 1) \beta}(t - \tau )^{-\frac{n}{2}\left(\frac{1}{r'} - \frac{1}{r}\right)} d \tau & = t^{- \beta ( \alpha + 1) - \frac{n}{2}\left(\frac{1}{r'} - \frac{1}{r}\right) + 1}\int_0^1 s^{-(\alpha + 1) \beta}(1 - s)^{-\frac{n}{2}\left(\frac{1}{r'} - \frac{1}{r}\right)} ds\nonumber\\
& =  t^{- \beta }\int_0^1 s^{-(\alpha + 1) \beta}(1 - s)^{-\frac{n}{2}\left(\frac{1}{r'} - \frac{1}{r}\right)} ds\nonumber\\
& = t^{-\beta}B\left(1 - \frac{n}{2}\left(\frac{1}{r'} - \frac{1}{r}\right), 1 - \beta(\alpha + 1)\right),
\end{align}
where $B(\cdot, \cdot)$ stands for the Beta function
$$B(\mu, \nu) = \int_0^1(1 - s)^{\mu - 1}s^{\nu - 1}ds, \ {\rm Re}(\mu), \ {\rm Re}(\nu) > 0.$$
Note that from \eqref{nalpha}, $1 - \frac{n}{2}\left(\frac{1}{r'} - \frac{1}{r}\right)$ and $ 1 - \beta(\alpha + 1)$ are positives.
Estimates \eqref{eq9gwp} and \eqref{eq10gwp} lead us to conclude
\begin{equation}\label{eq11gwp}
	\Big|\!\Big|\!\Big| \int_0^t S(t - \tau)(|x|^{-b}|u|^{\alpha}u)(\tau) d \tau\Big|\!\Big|\!\Big|
	\lesssim B\left(1 - \frac{n}{2}\left(\frac{1}{r'} - \frac{1}{r}\right), 1 - \beta(\alpha + 1)\right)\cdot M^{\alpha + 1}.
\end{equation}
Therefore
\begin{equation}\label{eq12gwp}
	|\!|\!| \Psi(u)|\!|\!| \leq \sup_{t > 0} t^{\beta}\| e^{it\Delta} \varphi \|_{L^{r, q}}  
	+ K M^{\alpha + 1},
\end{equation}
for some constant $K = K(\alpha, b) > 0$. Lets take $\rho$ sufficiently small and $M > 0$ so that
\begin{equation}\label{eq13gwp}
\rho + K M^{\alpha + 1} \leq M.
\end{equation}
Thus we get \eqref{welldef}. 

Next we shall prove the contraction \eqref{contraction}. The argument is similar to that we have just presented in the obtaining of \eqref{welldef}. In fact, consider $u, v \in X_M$. From Lemma \ref{l3p} and then Lemmas \ref{l1p} and \ref{l2p} it follows
\begin{align}\label{eq13.1gwp}
	\left\| \Psi (u) - \Psi (v)\right\|_{L^{r, q}} & \leq \left\| \int_0^t e^{i(t-\tau)\Delta}|x|^{-b}(|u|^{\alpha}u - |v|^{\alpha}v)(\tau) d \tau\right\|_{L^{r, q}} \nonumber\\
	& \lesssim \int_0^t (t - \tau)^{-\frac{n}{2}\left( \frac{1}{r'} - \frac{1}{r}\right)}\left\||x|^{-b}(|u|^{\alpha}u - |v|^{\alpha}v)(\tau)\right\|_{L^{r'\!, q}} d \tau\nonumber\\
	& \lesssim \int_0^t (t - \tau)^{-\frac{n}{2}\left( \frac{1}{r'} - \frac{1}{r}\right)}\left\||u|^{\alpha}u - |v|^{\alpha}v\right\|_{L^{\frac{r}{\alpha + 1}, \frac{q}{\alpha + 1}}} d \tau .
	\end{align}
Using that
$$||u|^{\alpha} u - |v|^{\alpha} v| \lesssim |u - v|^{\alpha + 1}$$
in \eqref{eq13.1gwp} and then Lemma \ref{l22p} we get
\begin{equation}\label{eq13.2gwp}
	\left\| \Psi (u) - \Psi (v)\right\|_{L^{r, q}}  
	 \lesssim \int_0^t (t - \tau)^{-\frac{n}{2}\left( \frac{1}{r'} - \frac{1}{r}\right)}\left\|u - v\right\|_{L^{r,q}}^{\alpha + 1} d \tau .
\end{equation}
Now, the same steps used to obtain \eqref{eq11gwp} lead us to the inequality
\begin{align}\label{eq13.3gwp}
	\left\| \Psi (u) - \Psi (v)\right\|_{L^{r, 2}}  
	&\lesssim t^{- \beta} B\left(1 - \frac{n}{2}\left(\frac{1}{r'} - \frac{1}{r}\right), 1 - \beta(\alpha + 1)\right)|\!|\!| u - v |\!|\!|^{\alpha + 1}\nonumber\\
	&\lesssim t^{- \beta}(2M)^{\alpha} B\left(1 - \frac{n}{2}\left(\frac{1}{r'} - \frac{1}{r}\right), 1 - \beta(\alpha + 1)\right)|\!|\!| u - v |\!|\!|.
\end{align}
That allows us to conclude
\begin{equation}\label{eq14gwp}
	|\!|\!| \Psi(u) - \Psi(v)|\!|\!| \leq (2M)^{\alpha} K|\!|\!| u - v |\!|\!|.
	\end{equation}
Choosing $M > 0$ so that
$$(2M)^{\alpha}K < 1,$$
we get the estimate \eqref{contraction}. With this we have the result.

\section{Self-similar solutions: Proof of Theorem \ref{selfsimilar}}\label{sec5}

As in \eqref{evohom}
\begin{align}
    t^{\beta}\|e^{it\Delta}\varphi\|_{L^{r,q}}=t^{\beta}t^{\frac{n-b}{2(\alpha+2)}-\frac{2-b}{2\alpha}}\|e^{i\Delta}\varphi\|_{L^{r,q}} = \|e^{i\Delta}\varphi\|_{L^{r,q}}.
\end{align}
Moreover, it follows from the conditions on $\alpha$ and $p$ we have $$0 < Re\,p < n$$ and $$r=\frac{n(\alpha+2)}{n-b}> max\left\{\frac{n}{Re\,p},\frac{n}{n-Re\,p}\right\}.$$ 
   Thus, from Proposition \ref{propcaz}, we get the finiteness of $\|e^{i\Delta}\varphi\|_{L^r}$. Then, since $L^r\hookrightarrow L^{r,q}$ continuously one has 
   $$\|e^{i\Delta}\varphi\|_{L^{r,q}}\leq \|e^{i\Delta}\varphi\|_{L^r}<\infty.$$
Then, for $\|e^{i\Delta}\varphi\|_{L^r}$ small enough, there exists a unique positively global solution $u$ of \eqref{pvi1} such that
\begin{align}\label{boundu}
\sup_{t>0}t^{\beta}\|u(t)\|_{L^{r,q}}\leq M.
\end{align}

The fact that the solution $u$ with initial value $\varphi$ is self-similar is a consequence of the uniqueness of the solution. Indeed, since $\lambda^p\varphi(\lambda x)=\varphi(x)$ for all $\lambda > 0$, the functions $\lambda^pu(\lambda x,\lambda^2t)$ are all solutions of \eqref{pvi1} with the same initial value $\varphi$ and all satisfying \eqref{boundu}.

\section{Scattering and Wave operator}\label{sec7}
\subsection{Proof of Theorem \ref{teo scattering} }
For $u_0\in\mathcal W$ we consider the corresponding global solution $u(t)$ of the IVP \eqref{pvi1} provided by Theorem \ref{globalWP}. We define 
\begin{equation}\label{eq S}
\psi_+(x):=u_0+i\gamma\dis\int_0^{+\infty}e^{-is\Delta}|x|^{-b}|u(x,s)|^\alpha u(x,s)ds.
\end{equation}

Applying the Schrödinger flow to both sides of \eqref{eq S} we find
\begin{align}
e^{it\Delta}\psi_+(x)&=e^{it\Delta}u_0+i\gamma\dis\int_0^{+\infty}e^{i(t-s)\Delta}|x|^{-b}|u(x,s)|^\alpha u(x,s)ds\\
&=u(t)+i\gamma\dis\int_t^{+\infty}e^{i(t-s)\Delta}|x|^{-b}|u(x,s)|^\alpha u(x,s)ds.
\end{align}

Using the same steps presented from \eqref{eq4gwp} to \eqref{eq9gwp} we get
\begin{align}
\|e^{it\Delta}\psi_+-u(t)\|_{L^{r,q}}&=\|\gamma\dis\int_t^{+\infty}e^{i(t-s)\Delta}|x|^{-b}|u(x,s)|^\alpha u(x,s)ds\|_{L^{r,q}}\\
&\lesssim\dis\int_t^{+\infty}|t-s|^{-\frac{n}{2}(\frac{1}{r'}-\frac{1}{r})} \||x|^{-b}|u(x,s)|^\alpha u(x,s)\|_{L^{r',q}}ds\\
&\lesssim|\!|\!|u|\!|\!|^{\alpha+1}\dis\int_t^{+\infty}|t-s|^{-\frac{n}{2}(\frac{1}{r'}-\frac{1}{r})}s^{-\beta(\alpha+1)}ds.
\end{align}

We note that, by using the change of variable $s=tr$ and the definition of $\beta$ (see \eqref{beta}) we can write 
\begin{align}
\dis\int_t^{+\infty}|t-s|^{-\frac{n}{2}(\frac{1}{r'}-\frac{1}{r})}s^{-\beta(\alpha+1)}ds&=t^{1-\frac{n}{2}(\frac{1}{r'}-\frac{1}{r})-\beta(\alpha+1)}\dis\int_1^{+\infty}(1-r)^{-\frac{n}{2}(\frac{1}{r'}-\frac{1}{r})}r^{-\beta(\alpha+1)}dr\nonumber\\
&=t^{-\beta}\dis\int_0^{1}(1-s)^{-\frac{n}{2}(\frac{1}{r'}-\frac{1}{r})}s^{\beta(\alpha+1)-2+\frac{n}{2}(\frac{1}{r'}-\frac{1}{r})}ds.
\end{align}
And now from the change of variable $r=1/s$ and the definition of  the beta function, we find 
\begin{align}
t^{-\beta}\dis\int_0^{1}(1-s)^{-\frac{n}{2}(\frac{1}{r'}-\frac{1}{r})}s^{\beta(\alpha+1)-2+\frac{n}{2}(\frac{1}{r'}-\frac{1}{r})}ds\lesssim t^{-\beta}B\left(1-\frac{n}{2}(\frac{1}{r'}-\frac{1}{r}),\beta(\alpha+1)+\frac{n}{2}(\frac{1}{r'}-\frac{1}{r})-1\right)
\end{align}
and then
\begin{equation}\label{eq6 scatt}
\dis\int_t^{+\infty}|t-s|^{-\frac{n}{2}(\frac{1}{r'}-\frac{1}{r})}s^{-\beta(\alpha+1)}ds \lesssim t^{-\beta}B\left(1-\frac{n}{2}(\frac{1}{r'}-\frac{1}{r}),\beta(\alpha+1)+\frac{n}{2}(\frac{1}{r'}-\frac{1}{r})-1\right).
\end{equation}
Therefore,
\begin{align}\label{estfinal}
\|e^{it\Delta}\psi_+-u(t)\|_{L^{r,q}}\lesssim t^{-\beta}B\left(1-\frac{n}{2}(\frac{1}{r'}-\frac{1}{r}),\beta(\alpha+1)+\frac{n}{2}(\frac{1}{r'}-\frac{1}{r})-1\right)|\!|\!|u|\!|\!|^{\alpha+1}.
\end{align}
From this we conclude $\psi_+\in\mathcal W$. Next we apply $e^{-it \Delta}$ to both sides of the integral equation
$$u(t) = e^{it \Delta}u_0 - i \gamma \int_0^te^{i(t - \tau) \Delta}|x|^{-b}|u|^{\alpha} u d\tau$$
to get
\begin{align}
e^{-it \Delta}u(t) &= u_0 - i \gamma \int_0^te^{-i \tau \Delta}|x|^{-b}|u|^{\alpha} u d\tau\\
&= \psi_+(x) + i \gamma \int_t^{+\infty}e^{- i \tau \Delta}|x|^{-b}|u|^{\alpha} u d\tau
\end{align}
and so
\begin{equation}\label{eq3 scatt}
    e^{-it \Delta}u(t) - \psi_+(x) = i \gamma \int_t^{+\infty}e^{- i \tau \Delta}|x|^{-b}|u|^{\alpha} u d\tau.
\end{equation}
Repeating the same estimates presented from \eqref{eq4gwp} to \eqref{eq7gwp} it follows
\begin{align}\label{eq4 scatt}
    \|e^{-it \Delta}u(t) - \psi_+\|_{L^{r, q}} &\leq c |\gamma| |\!|\!| u |\!|\!|^{\alpha + 1}\int_t^{+\infty} \tau^{-\frac{n}{2}(\frac{1}{r'}-\frac{1}{r} - \beta(\alpha + 1)} d\tau\nonumber\\
    & \leq c |\gamma| |\!|\!| u |\!|\!|^{\alpha + 1}\dfrac{t^{-\beta}}{\beta} .
\end{align}
From estimate \eqref{eq4 scatt} we conclude the proof of Theorem \ref{teo scattering}.
\subsection{Proof of Theorem \ref{teowo}}
For a fixed $\psi \in \mathcal{W}$ we consider the integral operator
\begin{equation}\label{int op wave}
Q_+(u)(t) := e^{it \Delta}\psi - i \gamma \int_t^{+\infty}e^{i(t - \tau) \Delta}|x|^{-b}|u|^{\alpha} u d\tau.
\end{equation}
For $u \in \mathcal{X}$ we argue as in \eqref{eq4gwp} to find
\begin{align}\label{eq1waveop}
\|Q_+(u)(t)\|_{L^{r, q}} &\leq \|e^{it \Delta}\psi\|_{L^{r, q}} +c |\gamma|\int_t^{+ \infty} (t - \tau )^{-\frac{n}{2}\left(\frac{1}{r'} - \frac{1}{r}\right)} \left\| |x|^{-b}|u|^{\alpha}u \right\|_{L^{r'\!,q}} d \tau\nonumber\\
&\leq \|e^{it \Delta}\psi\|_{L^{r, q}} + c |\gamma||\!|\!| u |\!|\!|^{\alpha + 1} \int_t^{+ \infty} (t - \tau )^{-\frac{n}{2}\left(\frac{1}{r'} - \frac{1}{r}\right)}\tau^{-\beta(\alpha + 1)} d \tau.
\end{align}
Using \eqref{eq6 scatt} we conclude
\begin{equation}\label{eq3waveop}
\|Q_+(u)(t)\|_{L^{r, q}} \leq  \|e^{it \Delta}\psi\|_{L^{r, q}}+ ct^{-\beta}|\!|\!|u|\!|\!|^{\alpha+1}B\left(1-\frac{n}{2}(\frac{1}{r'}-\frac{1}{r}),\beta(\alpha+1)+\frac{n}{2}(\frac{1}{r'}-\frac{1}{r})-1\right).
\end{equation}
Arguing as in \eqref{eq13.1gwp}, \eqref{eq13.2gwp} and \eqref{eq13.3gwp} we also have
\begin{align}\label{eq4waveop}
\|Q_+(u)(t) - Q_+(v)(t)\|_{L^{r, q}} & \lesssim |\gamma|\int_t^{+ \infty} (t - \tau )^{-\frac{n}{2}\left(\frac{1}{r'} - \frac{1}{r}\right)} \left\| |x|^{-b}(|u|^{\alpha}u - |v|^{\alpha}u )\right\|_{L^{r'\!,q}} d \tau\nonumber\nonumber\\
&\hspace{-1.4cm}\lesssim  |\gamma||\!|\!| u - v |\!|\!|^{\alpha + 1} \int_t^{+ \infty} (t - \tau )^{-\frac{n}{2}\left(\frac{1}{r'} - \frac{1}{r}\right)}\tau^{-\beta(\alpha + 1)} d \tau\nonumber\\
&\hspace{-1.4cm}\lesssim |\gamma|t^{-\beta}|\!|\!|u - v|\!|\!|^{\alpha+1}B\left(1-\frac{n}{2}(\frac{1}{r'}-\frac{1}{r}),\beta(\alpha+1)+\frac{n}{2}(\frac{1}{r'}-\frac{1}{r})-1\right).
\end{align}
From estimates \eqref{eq3waveop} and \eqref{eq4waveop} we conclude that $Q_+$ is a contraction map in some ball of $\mathcal{X}$ provided $\Psi$ is sufficiently small. So there exists $u\in \mathcal{X}$ such that 
\begin{equation}\label{solintwave}
u(t) = e^{it \Delta}\psi - i \gamma \int_t^{+\infty}e^{i(t - \tau) \Delta}|x|^{-b}|u|^{\alpha} u d\tau.
\end{equation}
Applying the Schr\"odinger operator $e^{-it\Delta}$ to both sides of \eqref{solintwave} we get
\begin{equation}\label{eq5waveop}
e^{-it \Delta}u(t)- \psi = - i \gamma \int_t^{+\infty}e^{-\tau \Delta}|x|^{-b}|u|^{\alpha} u d\tau.
\end{equation}
So, as in \eqref{eq4 scatt} we conclude
$$ \|e^{-it \Delta}u(t) - \psi\|_{L^{r, q}} \leq c |\gamma| |\!|\!| u |\!|\!|^{\alpha + 1}\dfrac{t^{-\beta}}{\beta} .$$
That finishes the proof of Theorem \ref{teowo}.

\section{Asymptotic stability: Proof of Theorem \ref{asymptsta}}\label{sec6}

First of all
\begin{equation}\label{eq1asymtheo}
u(t) - v(t) = e^{it\Delta}(\phi - \varphi) - \gamma \int_0^t e^{i(t-\tau)\Delta}|x|^{-b}(|u|^{\alpha}u - |v|^{\alpha}v)(\tau) d \tau.
\end{equation}
Then, 
\begin{align}\label{eq1asymtheo}
t^{\beta + h}\|u(t) - v(t)\|_{L^{r,q}} &\leq t^{\beta + h}\| e^{it\Delta} (\phi - \varphi)\|_{L^{r,q}}\nonumber\\ 
&+   t^{\beta + h} |\gamma|\int_0^t \| e^{i(t-\tau)\Delta}|x|^{-b}(|u|^{\alpha}u - |v|^{\alpha}v)(\tau)\|_{L^{r,q}} d \tau.
\end{align}
Following the steps presented in \eqref{eq13.1gwp} and \eqref{eq13.2gwp} one gets
\begin{align}\label{eq2asymtheo}
	t^{\beta + h}\left\| u(t) - v(t)\right\|_{L^{r, q}}  
	 & \leq t^{\beta + h}\left\| e^{it\Delta}( \phi - \varphi)\right\|_{L^{r, q}}\nonumber\\
  & + t^{\beta + h}|\gamma|\int_0^t (t - \tau)^{-\frac{n}{2}\left( \frac{1}{r'} - \frac{1}{r}\right)}\left\|u(\tau) - v(\tau)\right\|_{L^{r,q}}^{\alpha + 1} d \tau .
\end{align}
It turns out
\begin{align}\label{eq3asymptheo}
&t^{\beta + h}\int_0^t  (t - \tau)^{-\frac{n}{2}\left( \frac{1}{r'} - \frac{1}{r}\right)}\left\|u(\tau) - v(\tau)\right\|_{L^{r,q}}^{\alpha + 1} d \tau \nonumber\\
&= t^{\beta + h -\frac{n}{2}\left( \frac{1}{r'} - \frac{1}{r}\right) + 1}\int_0^1 (1 - s)^{-\frac{n}{2}\left( \frac{1}{r'} - \frac{1}{r}\right)}\left\|u(st) - v(st)\right\|_{L^{r,q}}^{\alpha + 1} ds \nonumber\\
&\lesssim \int_0^1 (1 - s)^{-\frac{n}{2}\left( \frac{1}{r'} - \frac{1}{r}\right)}s^{-\beta(\alpha + 1) - h}\left[(st)^{ \beta}\left\|u(st)\right\|_{L^{r,q}} + (st)^{ \beta}\left\|v(st)\right\|_{L^{r,q}}\right]^{\alpha} (st)^{\beta + h}\|u(st) - v(st)\|_{L^{r,q}}ds \nonumber\\
&\lesssim \big{(}|\!|\!| u |\!|\!| + |\!|\!| v |\!|\!|\big{)}^{\alpha}\int_0^1 (1 - s)^{-\frac{n}{2}\left( \frac{1}{r'} - \frac{1}{r}\right)}s^{-\beta(\alpha + 1) - h} (st)^{\beta + h}\|u(st) - v(st)\|_{L^{r,q}}ds.
\end{align}
Now, we claim that 
\begin{equation}\label{eq4asymptheo}
A:= \limsup_{t \rightarrow + \infty} t^{\beta + h} \|u(t) - v(t)\|_{L^{r,q}} < + \infty.
\end{equation}
Using estimate \eqref{eq3asymptheo} in \eqref{eq2asymtheo} and assuming the claim \eqref{eq4asymptheo} we arrive at the estimate
\begin{equation}\label{eq5asymptheo}
A \leq  \limsup_{t \rightarrow + \infty} \|e^{it\Delta}( \phi - \varphi)\|_{L^{r,q}} + c|\gamma|\left(|\!|\!| u |\!|\!| + |\!|\!| v |\!|\!|\right)^{\alpha}\left(\int_0^1 (1 - s)^{-\frac{n}{2}\left( \frac{1}{r'} - \frac{1}{r}\right)}s^{-\beta(\alpha + 1) - h}ds\right)A.
\end{equation}
From \eqref{Lasym}, the definition of  the beta function and \eqref{eq5asymptheo} we get
\begin{align}\label{eq6asymptheo}
A &\leq  c|\gamma|\left(|\!|\!| u |\!|\!| + |\!|\!| v |\!|\!|\right)^{\alpha} B\left(1 -\frac{n}{2}\left( \frac{1}{r'} - \frac{1}{r}\right), 1-\beta(\alpha + 1) - h\right)A\nonumber\\
&\leq c|\gamma|(2M)^{\alpha}B\left(1 -\frac{n}{2}\left( \frac{1}{r'} - \frac{1}{r}\right), 1-\beta(\alpha + 1) - h\right)A.
\end{align}
From the hypothesis \eqref{condasymp} we can assume
$$c|\gamma|(2M)^{\alpha}B\left(1 -\frac{n}{2}\left( \frac{1}{r'} - \frac{1}{r}\right), 1-\beta(\alpha + 1) - h\right) < 1$$
and then we conclude that $A = 0$. Therefore
$$\lim_{t \to + \infty} t^{\beta + h}\| u(t) - v(t)\|_{L^{r, 2}} = 0.$$


\begin{thebibliography}{32}
\providecommand{\natexlab}[1]{#1}
\providecommand{\url}[1]{\texttt{#1}}
\expandafter\ifx\csname urlstyle\endcsname\relax
  \providecommand{\doi}[1]{doi: #1}\else
  \providecommand{\doi}{doi: \begingroup \urlstyle{rm}\Url}\fi


\bibitem{AT21}
L.~Aloui and S.~Tayachi.
\newblock Local well-posedness for the inhomogeneous nonlinear {S}chr\"{o}dinger equation.
\newblock \emph{Discrete and Continuous Dynamic Systems}, 41(11): 5409-5437, 2021.

\bibitem{AC20}
A.~H. Ardila and M.~Cardoso.
\newblock Blow-up solutions and strong instability of ground states for the
  inhomogeneous nonlinear {S}chr\"{o}dinger equation.
\newblock \emph{Communications on Pure Applied Analysis}, 20\penalty0 (1):\penalty0 101--119,
  2021.


\bibitem{BL} 
J. Berg and J. L\"ofstr\"om. 
\newblock \textit{Interpolation Spaces.}
\newblock Springer-Verlag, New York, 1976.

\bibitem{BFV} 
P. Braz e Silva, L. Ferreira and E. Villamizar-Roa.
\newblock On the existence of infinite energy solutions for nonlinear Schrödinger equations. 
\newblock \textit{Proceedings of the American Mathematical Society.}  137 (6): 1977--1987, 2009.

\bibitem{Bourgain}
J. Bourgain.
\newblock \textit{Global solutions of nonlinear Schrödinger equations.} 
\newblock American Mathematical Society Colloquium Publications, vol. 46, American Mathematical Society, Providence, RI, 1999.

\bibitem{BFV09}
P. Braz e Silva, Lucas Ferreira, and Elder Villamizar-Roa.
\newblock On the existence of infinite energy solutions for nonlinear Schrödinger equations.
\newblock \emph{Proceedings of the American Mathematical Society}, 137(6):1977--1987, 2009.

\bibitem{C21}
L.~Campos.
\newblock Scattering of radial solutions to the inhomogeneous nonlinear
  {S}chr\"{o}dinger equation.
\newblock \emph{Nonlinear Analysis}, 202:\penalty0 112118, 17, 2021.

\bibitem{CCF22}
L. Campos, S. Correia, and L. G. Farah,
\newblock Some well-posedness and ill-posedness results for the INLS equation,
\newblock \emph{arXiv preprint arXiv:2210.07060}, 2022.


\bibitem{CF21}
M. Cardoso and L. G. Farah. 
\newblock Blow-up of radial solutions for the intercritical inhomogeneous NLS equation.
\newblock \emph{Journal of Functional Analysis}, 281(8): 109134, 2021.

	\bibitem[Cardoso, Farah, and
	Guzm{\'a}n]{CFG20}
	M.~Cardoso, L.~G. Farah, and C.~M. Guzm{\'a}n.
	\newblock On well-posedness and concentration of blow-up solutions for the intercritical inhomogeneous {NLS} equation.
	\newblock \emph{Journal of Dynamics and Differential Equations}, 35: 1337-1367, 2021.

\bibitem{CAZENAVEBOOK}
T. Cazenave, \emph{Semilinear Schrödinger equations}, Courant Lecture Notes in Mathematics, vol. 10, New York University, Courant Institute of Mathematical Sciences, New York; American Mathematical Society, Providence, RI, 2003.

\bibitem{CW} 
T. Cazenave and F. Weissler. 
\newblock Asymptotically self-similar global solutions of the nonlinear Schrödinger and heat equations. 
\newblock \textit{Mathematische Zeitschrift}. 228, 83--120, 1998.

\bibitem{CW00}
T. Cazenave and F. Weissler.
\newblock Scattering theory and self-similar solutions for the nonlinear Schrödinger equation.
\newblock \emph{SIAM Journal on Mathematical Analysis}, 31(3):625--650, 2000.


\bibitem{Cho-Lee}
Y. Cho and M. Lee.
\newblock On the orbital stability of inhomogeneous nonlinear {S}chrödinger equations with singular potential.
\newblock \emph{Bulletin of the Korean Mathematical Society}, 56 (5):1369--1383, 2019.


\bibitem{CG16}
V.~Combet and F.~Genoud.
\newblock Classification of minimal mass blow-up solutions for an {$L^2$}
  critical inhomogeneous {NLS}.
\newblock \emph{Journal of Evolution Equations}. 16\penalty0 (2):\penalty0 483--500, 2016.


\bibitem{DF05}
A.~De~Bouard and R.~Fukuizumi.
\newblock Stability of standing waves for nonlinear {S}chr\"{o}dinger equations
  with inhomogeneous nonlinearities.
\newblock \emph{Annales Henri Poincaré}. 6\penalty0 (6):\penalty0 1157--1177,
  2005.

\bibitem{dinh2017blowup}
V.~D. Dinh.
\newblock Blowup of {$H^1$} solutions for a class of the focusing inhomogeneous
  nonlinear {S}chr\"{o}dinger equation.
\newblock \emph{Nonlinear Analysis}, 174:\penalty0 169--188, 2018.

\bibitem{Dinh2017}
V. D. Dinh.
\newblock Scattering theory in a weighted $L^2$ space for a class of the defocusing inhomogeneous nonlinear Schrödinger equation.
\newblock \textit{Advances in Pure and Applied Mathematics}. 12(3): 38-72, 2021.


\bibitem{Farah}
L.~G. Farah.
\newblock Global well-posedness and blow-up on the energy space for the
  inhomogeneous nonlinear {S}chr\"{o}dinger equation.
\newblock \emph{Journal of Evolution Equation}, 16\penalty0 (1):\penalty0 193--208, 2016.


\bibitem{FG17}
L.~G. Farah and C.~M. Guzm\'{a}n.
\newblock Scattering for the radial 3{D} cubic focusing inhomogeneous nonlinear
  {S}chr\"{o}dinger equation.
\newblock \emph{Journal of  Differential Equations}. 262\penalty0 (8):\penalty0
  4175--4231, 2017.

\bibitem{GENOUD}
F. Genoud.
\newblock An inhomogeneous, {$L\sp 2$}-critical, nonlinear {S}chrödinger equation.
\newblock \emph{Zeitschrift für Analysis und ihre Anwendungen}, 31(3): 283--290, 2012.


	\bibitem{GENSTU}
	F.~Genoud and C.~A. Stuart.
	\newblock Schr\"{o}dinger equations with a spatially decaying nonlinearity:
	existence and stability of standing waves.
	\newblock \emph{Discrete and Continuous Dynamics Systems}, 21\penalty0 (1):\penalty0
	137--186, 2008.

\bibitem{Graf}
L. Grafakos, \emph{Classical Fourier Analysis}, 2nd ed., Grad. Texts in Math. {\bf 249}, Springer, New York, 2008.

\bibitem{CARLOS}
C. M. Guzm\'{a}n.
\newblock On well posedness for the inhomogeneous nonlinear {S}chr\"{o}dinger equation,
\newblock \emph{Nonlinear Analysis: Real World Applications}, 37: 249--286, 2017.

\bibitem{FELGUS}
F. Linares and G. Ponce.
\newblock \emph{Introduction to nonlinear dispersive equations}, 2nd ed., Universitext, Springer, New York, 2015.

\bibitem{kw1}
O. Kavian and F. B. Weissler,
\newblock The pseudo-conformally invariant nonlinear Schr{"o}dinger equation.
\newblock \emph{Michigan Mathematical Journal }, 41:151--173, 1994.

\bibitem{kw2}
O. Kavian and F. Weissler,
\newblock Remarks on the large time behaviour of a nonlinear diffusion equation,
\newblock \emph{Annales de l'Institut Henri Poincaré, section C}. 4 (5):423--452, 1987.

\bibitem{M21}
J.~Murphy.
\newblock A simple proof of scattering for the intercritical inhomogeneous
  {NLS}.
\newblock \emph{Proceedings of the American Mathematical Society}. 150: 1177-1186, 2022.

\bibitem{Zhang03}
C. Miao, B. Zhang, and X. Zhang.
\newblock Self-similar solutions for nonlinear Schr{"o}dinger equations.
\newblock \emph{Methods and applications of analysis}, 10(1): 119--136, 2003.

\bibitem{ON} 
R. O'Neil, 
\newblock Convolution operators and $L(p; q)$ spaces.
\newblock \textit{Duke Mathematical  Journal.} 30 (1): 129--142, 1963.

\bibitem{P00cauchy}
F. Planchon.
\newblock On the Cauchy problem in Besov spaces for a non-linear Schr\"odinger equation.
\newblock \emph{Communications in Contemporary Mathematics}, 2(02):243--254, 2000.

\bibitem{ribaud1998regular}
F. Ribaud and A. Youssfi,
\newblock Regular and self-similar solutions of nonlinear Schr{"o}dinger equations.
\newblock \emph{Journal de math{'e}matiques pures et appliqu{'e}es}, 77(10):1065--1079, 1998.

\bibitem{Tao}
T. Tao, \emph{Nonlinear dispersive equations}, CBMS Regional Conference Series in Mathematics, vol. 106, Published for the Conference Board of the Mathematical Sciences, Washington, DC; by the American Mathematical Society, Providence, RI, 2006. (Local and global analysis)
\end{thebibliography}
\end{document}